\documentclass[note,9pt,twoside,DIV=calc]{scrartcl}	%scrartcl switches from standard-documentclasses to Koma-Font, put in DIV=calc to make it calculate the div value
\usepackage{typearea}			%Package that allows to control textheigt and width
\areaset{16cm}{22cm}			%determines, using typearea, the height and widht

\input xy
\xyoption{all}
\usepackage{lscape}
\usepackage[latin1]{inputenc}
\usepackage[cmtip,arrow]{xy}
\usepackage{amsmath}
\usepackage{amssymb,amscd,amsthm,mathrsfs,yfonts,manfnt,pb-diagram,pb-xy}
\usepackage{hyperref}
\usepackage[nice]{nicefrac}		%Schreibt Brueche in leicht versetzter form, aufgerufen mit \nicefrac{}{}
\usepackage[T1]{fontenc}
\usepackage{lmodern} 					%this resolves the trouble with sumatra giving crappy output
\usepackage{tocenter}					%To fix the screwed up margins after putting twoside to have odd and even
\ToCenter[h]{\textwidth}{\textheight}	%This belongs to tocenter, and now centers all, h should mean headers centered too
\usepackage{ulem}							%To strike out or x-out stuff
\usepackage{color}

\renewcommand{\emph}{\textit}		%Tu emphasize text

\usepackage{ifthen}
\newboolean{comm}
\setboolean{comm}{false}
\newboolean{sol}
\setboolean{sol}{false}
\newboolean{notes}
\setboolean{notes}{false}			%%Own notes on the text
\CompileMatrices  						%saves xymatrices (that  is, xy-diagrams, those with \xymatrix{blabla} to speed up compiling, but 											changes them when they are changed

\newcommand{\com}{\ifthenelse{\boolean{comm}}}
\newcommand{\sol}{\ifthenelse{\boolean{sol}}}
\newcommand{\note}{\ifthenelse{\boolean{notes}}}

\newtheorem{Def}{Definition}
\newtheorem{Prop}[Def]{Proposition}

\newtheorem{Th}[Def]{Theorem}

\newtheorem{Lem}[Def]{Lemma}

\theoremstyle{definition}

\newcommand{\mZ}{\ensuremath{\mathbb{Z}}}
\newcommand{\mK}{\ensuremath{\mathbb{K}}}

\newcommand{\mc}{\mathcal}								%%mathcal
\renewcommand{\phi}{\varphi}

\newcommand{\mb}{\begin{pmatrix}}					%%End Matrix
\newcommand{\me}{\end{pmatrix}}						%%Begin Matrix
\newcommand{\ssu}{\ensuremath{\mathcal{S}}}						%%smooth suspension
\newcommand{\stylish}{}					%just redefine stylish to the style to change the \A and others at the same time

\newcommand{\A}{\ensuremath{\stylish{A}} }
\newcommand{\B}{\ensuremath{\stylish{B}} }

\usepackage{tocenter}				%To fix the screwed up margins after putting twoside to have odd and even
\ToCenter[h]{\textwidth}{\textheight}	%This now centers all, h should mean headers centered too

\setboolean{notes}{false}
\setkomafont{section}{\bfseries\Large}		%Sets the section-labeling to what you want

\usepackage{scrpage2}      	%Package to configure head and footlines that works well with Koma
\ofoot{}					
\ifoot{}
\rehead{\pagemark}			%Sets the pagenumber in left-even-head
\cehead{MARTIN GRENSING}	%writes in the center-even-head
\rohead{\pagemark}			%Sets the pagenumber in right-odd-head
\cohead{Stable homotopy theory}

\begin{document}

\title{\Large{\MakeUppercase{Noncommutative stable homotopy theory and $KK$-theory}}}
\author{\small{MARTIN GRENSING}}
\date{\small{\today}}
\maketitle

\thispagestyle{empty}
\begin{abstract} \small{Abstract: We construct  Kasparov's bifunctor $KK$ and $E$-theory by stable homotopy theoretic methods. This is motivated by  constructions of bivariant theories on more general categories such as, for example, bornological algebras. The details of the construction have interesting applications to groups of extensions.}
\end{abstract}
In  \cite{MR658514}, Rosenberg proposed to study monovariant $K$-theory from a stable homotopy point of view, and formulates the natural conjecture that Brown-Douglas-Fillmore $Ext$-theory and Kasparov's $KK$-theory should to some extend be describable in the stable homotopic framework. We will give a positive answer to this question, following ideas of Higson used to define  $E$-theory (\cite{MR1068250}); in particular, our constructions simplify the construction of $E$-theory by separating the stabilisation by compact operators from the stable-homotopy construction.

On the other hand, the bivariant $K$-functor of Kasparov (\cite{KaspOp}) is well known to be characterized  by certain exactness, stability and homotopy invariance properties. Cuntz has developed several versions of bivariant $K$-theories (\cite{MR2240217},\cite{MR1456322},\cite{MR2207702}) on the category of (complete) locally convex algebras, that posses desirable properties, referred to as $ckk$ in the sequel. An ultimate generalization has been given by in \cite{CuntzMeyer} to bornological algebras, based on the formalism of triangulated categories. 

On the other hand, in \cite{MR2964680}, we have shown, by a direct calculation using what we call locally convex Kasparov modules, that Bott periodicity holds for a larger class than $ckk$ is known to be universal for, namely for arbitrary split exact functors which are diffeotopy invariant and stable. The proof adapts readily to the $C^*$-setting. 

This makes it desirable to develop  universal split exact functors using only the Bott periodicity theorem. We here undertake this task for the category of $C^*$-algebras in a way that adapts readily to the categories of Banach, locally convex and Bornological algebras. At the same time, this extends $KK$ to not necessarily separable $C^*$-algebras.

We refer to \cite{MR2964680} for techniques allowing to calculate explicitly products of elements associated to Kasparov products in the category we construct here abstractly.
\newcommand{\ti}{X}

\section{Stable Homotopy Theory and Exactness}
We denote by $\mc{H}$ the homotopy category of $C^*$-algebras, i.e. the category with morphisms $\mc{H}(A,B):=C^*(A,B)/\sim$, where $\sim $ denotes homotopy.   The definitions of cones,  suspensions, cylinders, mapping cones and mapping cylinders are those from \cite{BlackK}. 
\begin{Def} Let $\mc{SH}$ be  the category with objects $\A_m$, where $\A$ is a $C^*$-algebras and $m\in \mZ$. We define the morphisms by
\[\mc{SH}(\A_m,\B_n):=\lim_{k\to\infty} \mc{H}(\ssu^{m+k}\A,\ssu^{n+k}\B)\]
where the connecting maps are given by the functor $\ssu$ and the limit is taken over all $k\in\mZ$ such that $m+k,n+k\geq 0$. 

$\Sigma$ denotes the functor defined by $\Sigma(\A_m):=\A_{m+1}$.
\end{Def}
We will frequently abbreviate $A_0$ to $A$. Note that as in the classical setting, the sets $\mc{SH}(A_m,B_n)$ carry natural group structures induced by concatenation. We denote by $X$ the class generated by $X^0$ and all isomorphisms in $\mc{SH}$.

We  denote by $X^0$ the collection of morphisms in of $C^*$-algebras that are inclusions of kernels of split surjections onto contractible algebras. 

Recall that Higson has given a (weak) set of axioms for a calculus of fractions in \cite{MR1068250}. We refer the reader to loc. cit. for the exact conditions on a class of morphisms in order to admit a calculus of fractions, and call such a set admissible. It is easy to see that Higson's axioms apply in the more general setting of categories which are not necessarily small, i.e., such that the isomorphism classes of objects form a set. In fact, it suffices to use (for example)  Grothendieck's framework of universes to carry over the theory from \cite{MR1068250}.

The fact that $X^0$ is an admissible set of morphisms follows essentially by the following two lemmas of independent interest: 

\begin{Lem}\label{stablepushout} Let 
\[\xymatrix{0\ar[r]&{I}\ar[r]^{\iota}&\A\ar[r]^\pi&\B \ar[r]&0}\]
be a split extension of $C^*$-algebras and $\phi:\mc{I}\to \mc{I}'$ a morphism. Then, up to homotopy equivalence,  there is a diagram with exact rows
\[\xymatrix{0\ar[r]&\ssu {I}\ar[r]^{\ssu \iota}\ar[d]_{\ssu \phi}&\ssu \A\ar[d]\ar[r]^{\ssu\pi}&\ssu \B\ar@{=}[d] \ar[r]&0\\
0\ar[r]&\ssu {I}'\ar[r]&\A'\ar[r]^{\pi'}&\ssu\B \ar[r]&0.
}\]
and where the bottom extension is split.
\end{Lem} 
\begin{proof}
The Lemma follows by replacing $SI'$ by the homotopy equivalent algebra which is the mapping cone of 
$$ev_0: CI\to I', f\mapsto \phi(f(0)).$$
\end{proof}

\begin{Lem} Let
\[  \xymatrix{0\ar[r]&{I}'\ar[r]_{\iota'}&\A'\ar[r]_{\pi'}&\B'\ar[r]&0}\]
be an extension of $C^*$-algebras with $B'$ contractible and $\pi$ split. Let $\phi:\A\to \A'$ be a morphism. Then up to homotopy equivalence there is a diagram with split exact rows
\[\xymatrix{0\ar[r]&I\ar[r]^{\iota}\ar[d]&\A\ar[d]^{\phi}\ar[r]^{\pi}& \B'\ar@{=}[d] \ar[r]&0\\
0\ar[r]& I'\ar[r]_{\iota'}&\A'\ar[r]_{\pi'}&\B'\ar[r]&0
}\]
\end{Lem}
\begin{proof} This follows by replacing $A$ with $Z_\phi$\end{proof}

\begin{Def} We denote by  $\mc{SH}(X^{-1})$ the localisation of $\mc{SH}$ with respect to $X$.\end{Def}
\begin{Th}  For every fixed $C^*$-algebra $D$ the functors $\mc{SH}(X^{-1})(D,\,\cdot\,)$ and $\mc{SH}(X^{-1})(\,\cdot\, ,D)$ are split exact. Every split exact, homotopy invariant Bott periodic functor factors uniquely over $\mc{SH}(X^{-1})$.
\end{Th}
\begin{proof}  It follows by standard arguments that for every morphism $\phi:A\to B$ one has exact sequences
\[\xymatrix{\mc{SH}(D,C_\phi)\ar[r]^{\pi_*}&\mc{SH}(D,A)\ar[r]^{\phi_*}&\mc{SH}(D,B)}\]
where $\pi:C_\phi\to \A$ is the canonical projection from the mapping cone sequence of $\phi$ (compare \cite{MR816237} or \cite{MR658514}). The image of this extension in $\mc{SH}(X^{-1})$ remains exact, as follows from results in \cite{MR1068250}. If $\phi$ is a split surjection, then one may deduce split exactness by replacing $C_\phi$ with $Ker(\phi)$. 
\end{proof}
All that remains is to stabilize the functor we have thus obtained. We denote by $\mK$ the compact operators on a separable Hilbert space.
\begin{Th}\label{refme} Define 
$$\tilde KK(A,B):=  \mc{SH}(X^{-1})(A\otimes\mK ,B\otimes\mK ).$$
Then $\tilde KK$ is naturally isomorphic to Kasparov's bivariant $K$-functor.
\end{Th}
\begin{proof}
$\tilde KK$ is stable because the inclusion $\mK\to \mK\otimes \mK$ obtained from a minimal projection is invertible up to homotopy. The usual characterization of $KK$ as a universal functor shows that we re-obtain $KK$.
\end{proof}
Using the same approach, but for the class $Y$ generated by all morphisms of $C^*$-algebras that are inclusions of kernels of surjections (not necessarily with a split) onto contractable algebras and all isomorphisms in  $\mc{SH}$, one obtains $E$-theory:
\begin{Th} Set
$$\tilde E(A,B):=\mc{SH}(Y^{-1})(A\otimes\mK ,B\otimes\mK ).$$
Then $\tilde E$ is naturally isomorphic to the functor $E$ defined in \cite{MR1068250} and \cite{MR1065438}.
\end{Th}
It is also easy to construct intermediate theories as in \cite{MR1667652}-our constructions replace the rather involved construction of the "noncommutative suspension $J$" by the straightforward suspension. 

Because $\tilde KK$ admits outer products one may now adopt the formal proof from \cite{kkT} to give a proof of the Thom isomorphism based only on Bott periodicity and without reference to the concrete description of $KK$: 
\begin{Prop} Let $H$ be a split exact, homotopy invariant and $\mK$-stable functor. Then 
\begin{align*} H(C_0(E))\approx H(C_0(X))
\end{align*}
for every complex bundle $E\to X$.
\end{Prop}
\section{Stable homotopy theory of extensions}
The Lemma \ref{stablepushout} further allows to define a theory of extensions based only on stable homotopy theory without stabilising by matrices.   In fact, for any two extensions
 
\[  \xymatrix{0\ar[r]&A\ar[r]&D_i\ar[r]&B\ar[r]&0}\]
of $B$ over $A$, we may define a sum by first passing to the pull-back  along the diagonal $B\to B\oplus B, \; b\mapsto (b,b)$ to obtain  
\[\xymatrix{0\ar[r]&A\oplus A\ar[r]&D_1\oplus D_2\ar[r]&B\oplus B\ar[r]&0}\]
Taking the stable homotopy pushout as in Lemma \ref{stablepushout} of the extension 
\[\xymatrix{ 0\ar[r]&S(A\oplus A)\ar[r]&S\big((D_1\oplus D_2)\oplus_\Delta B\big)\ar[r] \ar[r]&SB\ar[r]&0}.\]
 along the map $SA\oplus SA\to SA$ given by concatenation then defines a stable homotopy version of the  "Baer sum" of extensions.
\section{Bivariant $K$-theory for more general algebras}
The constructions outlined above carry over to other categories of algebras such as bornological, locally convex or Banach algebras. In these settings, one uses rather diffotopy and projective tensor products, together with appropriate versions of cylinders, cones and suspensions. 

The only point which differs significantly is the stabilisation. In the more general settings, following Cuntz, one may stabilize by smooth compact operators or the operator ideals $\mc{L}^p$ of $p$-summable operators.

However, it is more difficult to show that the functor obtained as in Theorem \ref{refme} is stable. One may instead "formally stabilize" in order to obtain a universal functor. If the inclusion $\mc{L}^p\to \mc{L}^p\otimes_\pi\mc{L}^p$ is invertible up to diffeotopy (a fact which seems very likely to be true, but which we have not been able to prove at present), then the two ways to stabilize can be seen to coincide.
\vspace{1cm}

Martin \textsc{Grensing}, \texttt{grensing@gmx.net},

Homepage: \href{http://www.martin.grensing.net}{Martin.Grensing.net},

D\'epartement de Math\'ematiques -- Universit\'e d'Orl\'eans, 

B.P. 6759 -- 45 067 Orl\'eans cedex 2, France.

\bibliographystyle{alpha}
\bibliography{../../../Fullbib}
\end{document}